\documentclass[12pt,a4paper]{amsart}

\usepackage{amsmath,amsthm,amssymb}
\usepackage{times}
\usepackage{enumerate}

\usepackage{verbatim}

\usepackage{color}

\usepackage{a4wide}

\theoremstyle{plain}

\usepackage{amsfonts}

\usepackage{amsmath}

\usepackage{amscd}
\usepackage{tabularx}
\usepackage{arydshln}
\usepackage{verbatim}
\usepackage{color}
\usepackage{a4wide}
 
\usepackage{hyperref}

\usepackage[all]{xy}
\usepackage{mathtools}

\theoremstyle{plain}

\usepackage[all]{xy}

\newtheorem{theorem}{Theorem}[section]
\newtheorem{lemma}[theorem]{Lemma}
\newtheorem{proposition}[theorem]{Proposition}
\newtheorem{corollary}[theorem]{Corollary}
\newtheorem{Counter-example}[theorem]{Counter-example}
\newtheorem{remark}[theorem]{Remark}

\theoremstyle{definition}

\newtheorem{definition}[theorem]{Definition}

\theoremstyle{remark}

\long\def\symbolfootnote[#1]#2{\begingroup\def\thefootnote{\fnsymbol{footnote}}
\footnote[#1]{#2}\endgroup}

\usepackage{amssymb}

\begin{document}

\def\Q{\mathbb Q}
\def\R{\mathbb R}
\def\N{\mathbb N}
\def\Z{\mathbb Z}
\def\C{\mathbb C}
\def\S{\mathbb S}
\def\L{\mathbb L}
\def\H{\mathbb H}
\def\K{\mathbb K}
\def\X{\mathbb X}
\def\Y{\mathbb Y}
\def\Z{\mathbb Z}
\def\E{\mathbb E}
\def\J{\mathbb J}
\def\I{\mathbb I}
\def\T{\mathbb T}
\def\H{\mathbb H}

\title{Codimension two spacelike submanifolds in \\ Lorentzian manifolds  and conformal structures}

\author{Rodrigo Mor\'on$^*$}
\address{
  Departamento de Matem\'{a}tica
  Aplicada, 
  Universidad de M\'{a}laga,  29071-M\'{a}laga (Spain)}
\email{ruyman@uma.es}

\email{fpalomo@uma.es}

\author{Francisco J. Palomo}



\date{}

\symbolfootnote[0]{

Both authors are  partially supported  by Spanish MICINN project PID2020-118452GB-100.
 
$^*$ Corresponding author.
 
}

\thispagestyle{empty}

\begin{abstract}
Starting from a Riemannian conformal structure on a manifold $M$, we provide a method to construct a family of Lorentzian manifolds. The construction relies on the choice  of a metric in the conformal class and a smooth  $1$-parameter family of self-adjoint tensor fields. Then, every metric in the conformal class corresponds to the induced metric on $M$ seen as a codimension two spacelike submanifold into these Lorentzian manifolds.
Under suitable choices of the $1$-parameter family of tensor fields,  there exists a lightlike normal  vector field along such spacelike submanifolds whose Weingarten endomorphism provide a M\"{o}bius structure on the Riemannian conformal structure.  Conversely,  every M\"{o}bius structure on a Riemannian conformal structure arises in this way.  
Flat M\"{o}bius structures are characterized in terms of the extrinsic geometry of the corresponding spacelike surfaces.
\end{abstract}

\maketitle

\markboth{}{}

\hyphenation{Lo-rent-zi-an}

\section{Introduction}

\noindent A Riemannian conformal structure on a manifold $M$ is an equivalence class of Riemannian metrics on $M$ where two metrics are equivalent 
if they differ by a factor that is a smooth positive function on the manifold $M$.  Conformal structures (in Lorentzian signature) was introduced by Hermann Weyl in order to formulate a unified fields theory. Weyl wrote `` 
To derive the values of the quantities $g_{ik}$ from directly
observed phenomena, we use light-signals ....
By observing the arrival of light at the points neighbouring to $O$
we can thus determine the ratios of the values of the $g_{ik}$’s ..... It is impossible,
however, to derive any further results from the phenomenon of
the propagation of light..." \cite[Chap. 4, Sec. 27]{Weyl}.

From a mathematical perspective,  the problem of the equivalence for conformal structures on $(n\geq 3)$-dimensional manifolds  was solved by E. Cartan by means of the now called canonical normal Cartan connection \cite{Car23}. For dimension $n\geq 3$,  Riemannian conformal structures $(M,c)$ correspond bijectively (up to isomorphism)  with  normal Cartan geometries of type $(G,P)$ where $G=O(1,n+1)/\{\pm \mathrm{Id}\}$ is the  M\"{o}bius  group and $P$ is the Poincar\'{e} conformal group defined to be the isotropy group of the line through an isotropic (lightlike) vector (see details in  \cite[Theor. 1.6.7]{CS09}). That is, conformal structures on  an $(n\geq 3)$-dimensional manifold $M$ gives rise to a principal $P$-bundle $\mathcal{P}\to M$ and a unique Cartan connection $\omega \in \Omega^{1}(\mathcal{P}, \mathfrak{g})$ where $\mathfrak{g}$ is the Lie algebra of the  M\"{o}bius  group $G$ such that $\omega$ satisfies certain normalization conditions  and conversely.

 These results have been extended to dimensions one and two by means of the notion of M\"obius structure \cite{Burs},  Section \ref{prel}. A M\"{o}bius structure on a manifold $M$ is essentially equivalent to defining a conformal class of metrics $c$ on $M$ and a “Schouten type-tensor” for $c$, that is, 
 a map $$
	D:c \to \mathcal{T}_{(0,2)}M
	$$ 
	such that for every $g\in c$ the tensor $D(g)$ is symmetric with $\mathrm{trace}_{g}D(g)=\frac{\mathrm{scal}^{g}}{2(n-1)}$ and $D$
follows the same conformal transformation law that the Schouten tensor, Definition \ref{280221C}. Thus, the notion of M\"{o}bius structure provides a uniform description of Cartan geometries of type $(G,P)$ for all dimensions. More explicitly, if we start with a conformal structure $(M,c)$ on an $(n\geq 2)$-dimensional manifold $M$, the set of Cartan geometries of type $(G,P)$ is in one-to-one correspondence with the set of “Schouten type-tensor” for $c$. Hence, this notion is specially relevant for conformal structures on surfaces. 

The planning of this paper is as follows. Starting from a Riemannian conformal structure $(M,c)$, by setting a metric $g\in c$ and an admissible  $1$-parameter family $\alpha \colon \R \to \mathcal{T}_{(1,1)}M$, see Definition \ref{060923A}, we construct a $(n+2)$-dimensional Lorentzian manifold $(\widetilde{M}, \widetilde{g})$, see Proposition \ref{280321D}, such that
\begin{enumerate}
\item there is a distinguished lightlike hypersurface $\mathcal{Q}\subset \widetilde{M}$ (see Definition \ref{210721A}) and
\item every metric in the conformal class $e^{2u}g\in c$
is the induced metric of an immersion from $M$ to $\widetilde{M}$ through $\mathcal{Q}$. Such immersions are defined in (\ref{immer}) and are denoted by $\Psi^u$.
\end{enumerate}

\noindent This construction 
is inspired by the Fefferman and Graham ambient metric for conformal structures in the 1980s, \cite{FG85} (see also \cite{FG}).
 Roughly speaking, starting with a Riemannian conformal structure $(M,c)$,  the space of scales $\mathcal{Q}$ consists of the rays of metrics $y:=t^{2}g_{x}$ on $T_{x}M$ where $x\in M$, $t\in \R^{+}$ and $g\in c$. The ambient metric $\widetilde{g}$ is defined so that $(\widetilde{M}, \widetilde{g})$ is a Lorentzian manifold  that admits $\mathcal{Q}$ as an embedded lightlike hypersurface. The original Fefferman-Graham metric requires certain normalisation condition  (see Remark \ref{050222}). 
In this paper,  we will adopt the weaker  notion of pre-ambient space given  in \cite{Cap},   Definition \ref{0706211}. The pre-ambient metric $\widetilde{g}$ that we define in (\ref{220221A}) is not a warped product metric in general, Remark \ref{280221A}. 

Now, every spacelike immersion $\Psi^{u}$ has codimension two in $(\widetilde{M}, \widetilde{g})$ and its normal bundle is spanned by the lightlike vector fields
vector fields $\xi^{u}$ and $\eta^{u}$ given in (\ref{normal vectors}). The main aim of this paper is to show  Theorem \ref{040222} which states that:
\begin{quote}
Assume  the admissible  $1$-parameter family $\alpha$ satisfies $\mathrm{trace}(\alpha(0))=\frac{\mathrm{scal}^{g}}{n-1}$ where $\mathrm{scal}^{g}$ is the scalar curvature of the fixed metric $g$. Then, the assignment
		$$
		D:c\to \mathcal{T}_{(0,2)}M, \quad e^{2u}g\mapsto e^{2u}g\left(A_{\eta^u}(-),- \right),
		$$
 defines a   M\"{o}bius structure for the Riemannian conformal structure $(M, c)$, where $A_{\eta^{u}}$ denotes the Weingarten endomorphism of  $\eta^{u}$. Moreover, every M\"{o}bius structure for a Riemannian conformal structure $(M, c)$ arises in this way.
\end{quote}

The content of this paper is distributed as follows. In Section \ref{prel}, taking into account ideas from \cite{Burs} and \cite{calderbank}, we recall the notion of M\"{o}bius structure on Riemannian conformal structures $(M,c)$ as was introduced in \cite{burden2015numerical2}, extended to arbitrary dimension in an obvious way.   We also include several basic facts on spacelike submanifolds in Lorentzian geometry. The spacelike submanifolds have been studied for a long time, both from the physical and mathematical points of view (see for instance \cite{Rom} and references therein). Then, we  show  some properties from the Lorentzian geometry perspective of the notion of pre-ambient space. 
Section \ref{ambient metrics} provides an explicit method to construct examples of pre-ambient spaces and includes several curvature properties of these pre-ambient spaces. 
In particular, we give conditions which permit to assure that the Ricci tensor of these pre-ambient spaces vanishes along  $\mathcal{Q}$, Corollary \ref{ricci}.

The main results are in Section \ref{main results} where it is essentially  shown that M\"{o}bius  structures agree with certain Weingarten endomorphisms of codimension two spacelike submanifolds in these pre-ambient spaces, Theorem \ref{040222}. 
This result is remarkable for conformal structures in surfaces. As was mentioned, there is no preferred  M\"{o}bius structure on a $2$-dimensional Riemannian conformal structure. Theorem \ref{040222} provides an explicit method to construct such structure. 
We hope that our viewpoint sheds some light on the interplay between the theory of spacelike submanifolds and  M\"{o}bius structures  on  Riemannian conformal structures. Section \ref{main results} also includes several properties on the family of spacelike immersions we need to construct the M\"{o}bius  structure. In fact, Corollary \ref{240321A} shows that the normal curvature tensor of such immersions always vanishes. Also, as a consequence of Remark \ref{100321A}, the mean curvature vector field of the isometric immersion $\Psi^u$ with induced  metric $e^{2u}g$ satisfies
$$\|\mathbf{H}^u\|^2=\frac{\mathrm{scal}^{e^{2u}g}}{n(n-1)},$$
see details in Remark \ref{k=H}. Particular cases of this formula have been previously obtained  in \cite[Cor. 4.5]{PPR} and \cite[Cor. 3.7]{PR13}.
Note that the causal character of $\mathbf{H}^u$ in the Lorentzian manifold $\widetilde{M}$ is determined by the sign of the scalar curvature of the metric $e^{2u}g.$
Remark \ref{k=H} also includes that $\nabla^{\perp}\mathbf{H}^{u}=0$ if and only if  $\mathrm{scal}^{e^{2u}g}$ is constant (compare with \cite[Cor. 3.10]{PR13}). In particular, when $M$ is compact, the positive answer  to the Yamabe problem implies that there exists an immersion $\Psi^{u}$  with parallel mean curvature vector field.

Section \ref{app} focusses in the two dimensional case, we write down the Codazzi equation in terms of the Cotton-York tensor, Lemma \ref{160721C}. Then, Proposition \ref{290321D} shows that tangent spaces of $M$ along these immersions are invariant under the curvature tensor of $(\widetilde{M}, \widetilde{g})$ if and only if the Cotton-York tensor of $c$ vanishes. 
In the terminology of \cite{calderbank}, \cite{burden2015numerical2}, this means that
 the  M\"{o}bius structure $D$ on $(M, c)$ is flat.

\section{Preliminaries}\label{prel}

\noindent All the manifolds are assumed to be smooth, Hausdorff,  satisfying the second axiom of countability and without boundary.  Let $M$ be a manifold with $\mathrm{dim}M =n \geq 2$. A Riemannian
conformal structure on $M$ is an equivalence class $c=[g]$ of Riemannian metrics  where two metrics $g$ and $g'$ are said to be equivalent when $g'=e^{2u}g$ for a smooth function $u$ on $M$.

A M\"{o}bius  structure on a manifold $M$ is essentially equivalent to defining a confomal class of metrics $c$ on $M$ and a “Schouten type-tensor” for $c$. 
This problem was addressed in \cite{calderbank} and \cite[Sec. 5]{Burs}.  For  our purposes, we adopt the following definition.

\begin{definition}\label{280221C}(\cite{calderbank}, \cite{burden2015numerical2})
A M\"{o}bius structure on an $(n\geq 2)$-dimensional manifold $M$ is a triple $(M , c, D)$  where $c$ is a Riemannian conformal structure on $M$ and 
\begin{enumerate}
\item $D$ is a map
	$
	D:c \to \mathcal{T}_{(0,2)}M
	$
	such that for every $g\in c$, the tensor $D(g)$ is symmetric with $$\mathrm{trace}_{g}D(g)=\frac{\mathrm{scal}^{g}}{2(n-1)},$$   where $\mathrm{scal}^{g}$ is the scalar curvature of the metric $g\in c$  and $\mathrm{trace}_{g}D(g)$ denotes the $g$-metric trace of the corresponding tensor $D(g).$

\item $D$ satisfies the following conformal transformation law
	\begin{equation}\label{030221A}
	D(e^{2u}g)=D(g)- \frac{\| \nabla^g u \|^{2}_{g}}{2}g- \mathrm{Hess}^{g}(u)+ du \otimes du,
	\end{equation}
	where $ \nabla^g u$ and $\mathrm{Hess}^{g}(u)$ are the gradient and the Hessian of the function $u\in \mathcal{C}^{\infty}(M)$ for the metric $g$, respectively.

\end{enumerate}
\end{definition}
\noindent We mean the map
$D$ as a  M\"{o}bius structure for the conformal structure $c$.  The conformal transformation law  implies that a M\"{o}bius structure $D$ for a conformal class $c$ is completely determined by the value at a single $g\in c$. In fact, the relationship between the scalar curvatures of two conformally related metrics and the conformal transformation law  imply that $\mathrm{trace}_{e^{2u}g}\,D(e^{2u}g)=\frac{\mathrm{scal}^{e^{2u}g}}{2(n-1)}$.

\smallskip

For $(n\geq 3)$-dimensional Riemannian conformal structures $(M,c)$,  there is a preferred  M\"{o}bius  structure.  In fact, let us recall that  Schouten tensor is defined by
$$
P^{g}(X,Y)= \frac{1}{n-2}\Big(\mathrm{Ric}^{g}(X,Y)-\frac{\mathrm{scal}^{g}}{2(n-1)}g \Big),
$$
where $\mathrm{Ric}^{g}$ denotes the Ricci tensor of the Riemannian metric $g\in c$.
The well-known conformal transformation law for the Schouten tensor implies that $D(g)=P^{g}$ provides a M\"{o}bius structure for the conformal class $c$.
 Therefore,   for conformal structures on $(n\geq 3)$-dimensional manifolds,  the  Schouten tensor gives a canonical M\"{o}bius  structure.  
For  the two dimensional case, there is something new.  Namely,  on a $2$-dimensional conformal Riemannian manifold $(M,c)$, a M\"{o}bius structure  is equivalent  to specifying a  “Schouten type-tensor”,  \cite{Burs}, \cite{calderbank}.

\begin{remark}\label{190721A}
{ \rm
The Uniformization Theorem  states that a $2$-dimensional Riemannian manifold $(M,g)$ admits a metric $g'$ conformal to $g$ with constant Gauss curvature $k$. This fact leads to a choice of the  M\"{o}bius structure determinated by $D(g')=(k/2) \, g'$ and the conformal transformation law. On the other hand,  recall that for a connected oriented $2$-dimensional manifold $M$, there is a well-known one-to-one correspondence between conformal classes and complex structures.  A Riemann surface is a such $2$-dimensional manifold endowed with a particular choice of conformal or complex structure.   Thus,  a  M\"{o}bius  structure on a connected oriented $2$-dimensional manifold $M$ is equivalent to specifying a complex structure and a “Schouten type-tensor” on $M$.
}
\end{remark}

\begin{remark}
{ \rm For $n\geq 3$ and
taking into account $2\,\mathrm{div}\, \mathrm{Ric}^{g}=d\, \mathrm{scal}^g$ (see for instance \cite[Cor. 3.54]{One83}),  one gets that $\mathrm{div}\, \mathrm{P}^{g}=\frac{1}{2(n-1)}d\, \mathrm{scal}^g.$ This property is not satisfied for M\"{o}bius structures,   in general.

}
\end{remark}

\smallskip

In this section we also fix some terminology and notations for spacelike immersions in Lorentzian manifolds. Let $(\widetilde{M}, \widetilde{g})$ be an $(m\geq 2)$-dimensional  Lorentzian manifold. That is, $(\widetilde{M}, \widetilde{g})$  is a semi-Riemannian manifold endowed with a metric tensor $\widetilde{g}$ of signature $(1, m-1)$.  A smooth immersion  $\Psi:M\rightarrow (\widetilde{M}, \widetilde{g})$ of a (connected) $n$-dimensional  manifold $M$  is said to be spacelike when the induced metric $g:= \Psi^{*}(\widetilde{g})$ is Riemannian.

Let $\overline{\mathfrak{X}}(M)$ be the $C^{\infty}(M)-$module of vector fields along the spacelike immersion $\Psi$.  Every vector field $X \in \mathfrak{X}(\widetilde{M})$ provides, in a natural way,  the vector field $X\mid_{\Psi}:= X \circ \Psi \in \overline{\mathfrak{X}}(M)$.
As usual, for $V\in \overline{\mathfrak{X}}(M)$, we have the decomposition $V=V^{\top}+V^{\bot}$, where $V^{\top}_{x}\in T_{x}\Psi\cdot T_{x}M$ and $V^{\bot}_{x}\in (T_{x}\Psi\cdot T_{x}M)^{\bot}$ for all $x\in M$. We have agreed to denote by $T \Psi$  the differential map of $\Psi$.
We call $V^{\top}$ the tangent part of $V$ and $V^{\bot}$ the normal part of $V$. The $C^{\infty}(M)-$submodule of $\overline{\mathfrak{X}}(M)$ of all normal vector fields along $\Psi$ is denoted by $\mathfrak{X}^{\perp}(M)$, that is, $\mathfrak{X}^{\perp}(M)=\{V \in \overline{\mathfrak{X}}(M): V^{\top}=0\}$. The set of vector fields $\mathfrak{X}(M)$ may be seen as a $C^{\infty}(M)-$submodule of $\overline{\mathfrak{X}}(M)$ by meaning of 
$$
\mathfrak{X}(M) \to \overline{\mathfrak{X}}(M),\quad  V \mapsto T\Psi \cdot V,
$$
where $(T \Psi \cdot V)(x):=T_{x}\Psi \cdot V_{x}$ for all $x\in M$. 
In order to avoid ambiguities,  we explicitly write the differential map $T\Psi$ when necessary.

We write $\nabla^g$ and $\widetilde{\nabla}$ for the Levi-Civita connections of $(M,g)$ and $(\widetilde{M}, \widetilde{g})$, respectively. 
As usual, we also denote by $\widetilde{\nabla}$ the induced connection and  by $\nabla^{\perp}$ the normal connection on $M$. 
The decomposition of the induced connection $\widetilde{\nabla}$, into tangent and normal parts, leads to the Gauss and Weingarten formulas of $\Psi$ as follows
\begin{equation}\label{shape}
\widetilde{\nabla}_V W=T \Psi \cdot \nabla_V^{g}W + \mathrm{II}(V,W) \quad \quad \mathrm{and} \quad \quad \widetilde{\nabla}_V\xi=-T \Psi\cdot A_{\xi}V+\nabla^{\perp}_V\,\xi,
\end{equation}
for every tangent vector fields $V,W\in\mathfrak{X}(M)$ and $\xi\in\mathfrak{X}^{\perp}(M)$. Here $\mathrm{II}$ denotes the second fundamental form and $A_{\xi}$ the Weingarten endomorphism (or shape operator) associated to $\xi$. 
For vector fields $U,V,W\in \mathfrak{X}(M)$, we let
\begin{equation}\label{160721A}
(\nabla_{U}\mathrm{II})(V, W)=\nabla^{\perp}_{U}(\mathrm{II}(V, W))- \mathrm{II}(\nabla_{U}V, W)- \mathrm{II}(V, \nabla_{U}W).
\end{equation}
Then the Codazzi equation reads as follows (see for instance \cite[Prop. 4.33]{One83},  taking into account that our convention on the sign of the Riemannian curvature tensor is the opposite to \cite{One83})
\begin{equation}\label{160721B}
(\nabla_U \mathrm{II})(V,W)-(\nabla_V \mathrm{II} )(U,W)= \left(\widetilde{R}(T\Psi \cdot U,T\Psi \cdot V)T\Psi \cdot W\right)^{\perp},
\end{equation}
where $\widetilde{R}$ is the curvature tensor of $\widetilde{\nabla}$.
Every Weingarten endomorphism $A_{\xi}$ is self-adjoint and the second fundamental form is symmetric. They are also related by the following formula
\begin{equation}\label{230321C}
g\left(A_{\xi}V,W\right)  = \widetilde{g} \left(\mathrm{II}(V,W), \xi \right).
\end{equation}
The normal curvature tensor $R^{\perp}$ is given by
$$
R^{\perp}(V,W)\xi=\nabla^{\perp}_{V}\nabla^{\perp}_{W}\xi-\nabla^{\perp}_{W}\nabla^{\perp}_{V}\xi- \nabla^{\perp}_{[V,W]}\xi
$$
and the mean
curvature vector field  by
$\mathbf{H}=\frac{1}{n}\,\mathrm{trace}_{g}\mathrm{II}.$

A particular case occurs when, working with a codimension two immersion $\Psi$, we are able to find a global  lightlike normal frame $\{\xi,\eta\}$ along $\Psi$. That is, $\xi$ and $\eta$ are two globally defined normal vector fields  along $\Psi$ which are lightlike (i.e., $\widetilde{g}(\xi, \xi)=\widetilde{g}(\eta, \eta)=0$) with the normalization condition $\widetilde{g}(\xi,\eta)=-1$. Let $A_{\xi}$ and $A_\eta$ be the associated Weingarten endomorphisms. Then, for every $V,W\in\mathfrak{X}(M)$, the second fundamental form can be written as
\begin{equation}\label{270221B}
\mathrm{II}(V,W)=-g(A_{\eta}V,W)\xi-g(A_{\xi}V,W)\eta.
\end{equation}
Taking traces in this expression, we obtain for the mean curvature vector field
\begin{equation}\label{H}
\mathbf{H}=-\dfrac{1}{n}\left(\mathrm{trace}(A_{\eta})\xi+\mathrm{trace}(A_{\xi})\eta\right).
\end{equation}

\bigskip

Let $(M,c)$ be a Riemannian conformal structure on an $(n\geq 2)$-dimensional manifold $M$. 
Let us consider the $\R^{+}$-principal fiber bundle $\pi: \mathcal{Q}\to M$ defined as the ray fiber subbundle in the fiber bundle of Riemannian metrics given by metrics in the conformal class $c$. Thus, the fiber over $x\in M$ is formed by the values of $g_{x}$ for all metrics $g\in c$. 
 Every section of $\pi$ provides a Riemannian metric in the conformal class $c$ and the principal $\R^{+}$-action on $\mathcal{Q}$ is given by $\varphi(\tau, g_{x})=\tau^{2}g_{x}$, $x\in M$. 
 Let us denote by $Z_{\mathcal{Q}}$ the fundamental vector field for the action $\varphi$, that is,
 $$
 Z_{\mathcal{Q}}(g_{x})=\left.\frac{d}{d t}\right |_{t=0}\varphi (e^{t}, g_{x})=\left.\frac{d}{d t}\right |_{t=0}(e^{2t}g_{x}).
 $$
The principal bundle $\pi: \mathcal{Q}\to M$ is called the scale bundle of $(M,c)$.

\begin{definition}\label{0706211}(\cite{FG})
A pre-ambient space  for a Riemannian conformal structure $(M, c)$ is an $(n+2)$-dimensional Lorentzian manifold $(\widetilde{M}, \widetilde{g})$ such that
\begin{enumerate}

\item There is a free $\R^{+}$-action $\widetilde{\varphi}$ on $\widetilde{M}$ and an embedding  $\iota: \mathcal{Q}\to \widetilde{M}$ such that the following diagram commutes
\begin{equation}\label{ambientdiag}
\begin{CD}
\R^{+}\times \mathcal{Q} @ >  \mathrm{id}_{\R^{+}}\times \iota >> \R^{+}\times \widetilde{M} \\
@V  \varphi VV @VV  \widetilde{\varphi} V \\
\mathcal{Q} @>>\iota  > \widetilde{M}
\end{CD}
\end{equation} 

\noindent Hence, the fundamental vector field $Z\in \mathfrak{X}(\widetilde{M})$ for the action $\widetilde{\varphi}$ and 
the vector field $Z_{\mathcal{Q}} \in \mathfrak{X}(\mathcal{Q})$ are $\iota$-related, i.e.,
 $T_{g_{x}}\iota \cdot  Z_{\mathcal{Q}}(g_{x})=Z(\iota (g_{x}))
 $
for all $g_{x}\in \mathcal{Q}$.

\item For $Z\in \mathfrak{X}(\widetilde{M})$,  we have $\mathcal{L}_{Z}\widetilde{g}= 2 \widetilde{g},$ where $\mathcal{L}$ is the Lie derivative.

\item For any $g_{x}\in \mathcal{Q}$ and $\xi, \eta \in T_{g_{x}}\mathcal{Q}$, the following equality holds
$$
\iota^{*}(\widetilde{g})_{g_{x}}(\xi, \eta)= g_{x}(T_{g_{x}}\pi \cdot \xi,T_{g_{x}}\pi \cdot \eta ).
$$
In particular, we have $\iota^{*}(\widetilde{g})(Z_{\mathcal{Q}},-)=0.$
\end{enumerate}
\end{definition}

\noindent For a pre-ambient space $(\widetilde{M}, \widetilde{g})$ the metric $\widetilde{g}$ is called a pre-ambient metric.
The condition $\mathcal{L}_{Z}\widetilde{g}= 2 \widetilde{g}$ tells us that the vector field $Z$ is 
homothetic  with respect to the pre-ambient metric $\widetilde{g}$. 
\begin{remark}\label{050222}
{\rm The notion of ambient metric in \cite{FG} satisfies a normalisation condition. In fact, in order to obtain the uniqueness of the ambient Lorentzian metric $\widetilde{g}$, the ambient metric by Fefferman and Graham
imposes that  the Ricci tensor of the metric $\widetilde{g}$ vanishes to a certain order
(depending on the dimension) on $\mathcal{Q}$, see \cite{FG} for details. 
The pre-ambient space has been used by {\v C}ap and Gover in order to get the relationships to
 the standard tractors, see \cite{Cap}. 
}
\end{remark}

We  end this section with several comments from the point of view of  Lorentzian geometry  of the notion of pre-ambient space. 
\begin{definition}\label{210721A}
A lightlike manifold is a pair $(N,h)$ where $N$ is an $(n+1)$-dimensional smooth manifold with $n\geq 2$ and furnished with a lightlike metric $h$. That is, $h$ is a symmetric $(0,2)$-tensor field on $N$ such that
\begin{enumerate}
\item $h(\xi,\xi)\geq 0$ for all $\xi\in \mathfrak{X}(N)$. 
\item For every $y\in N$, the radical $\mathrm{Rad}(h)(y)=\{\xi\in T_{y}N: h(\xi,-)=0\}$   defines a $1$-dimensional distribution on $N$.
\end{enumerate}
A smooth immersion $\Psi\colon N^{n+1} \to (\widetilde{M}^{n+2}, \widetilde{g})$ in an arbitrary Lorentzian manifold is said to be a lightlike hypersurface when the induced tensor $\Psi^{*}(\widetilde{g})$ is a lightlike metric.

\end{definition}

\smallskip
Now, let $(M,c)$ be a Riemannian conformal structure on an $(n\geq 2)$-dimensional manifold $M$ and $(\widetilde{M}, \widetilde{g})$ a pre-ambient space  for $(M,c)$. Then, condition $(3)$ in Definition \ref{0706211} implies that $\iota\colon\mathcal{Q}\to \widetilde{M}$ is a lightlike hypersurface. Moreover,  the induced lightlike metric $h: =\iota^{*}(\widetilde{g})$ 
 does not depend on the particular pre-ambient metric $\widetilde{g}$. In the terminology of \cite{FG}, the lightlike metric $h$ is called the tautological tensor. The radical distribution $\mathrm{Rad}(h)$ is globally generated by the vector field $Z_{\mathcal{Q}}$.

 \smallskip

Recall that every choice of a metric $g\in c$ provides a section of $\pi \colon \mathcal{Q}\to M$ and conversely. The following result is well-known. We include here a proof for the sake of completeness.
\begin{lemma}\label{endomorfismo}
Let $(M,c)$ be a Riemannian conformal structure and $(\widetilde{M}, \widetilde{g})$ a pre-ambient space for $(M,c)$. For every $g\in c$, the map
$$
\Psi^{g}:=\iota \circ g\colon M\to (\widetilde{M},\widetilde{g})
$$
is a codimension two spacelike immersion with induced metric $ (\Psi^{g})^{*}(\widetilde{g})=g$. Moreover, the vector field $\xi:=Z\mid_{\Psi^g} $ is normal and lightlike along $\Psi^g$ with $A_{\xi}=-\mathrm{Id}.$
\end{lemma}
\begin{proof}
For every $x\in M$ a direct computation gives
\[
(\Psi^{g})^{*}(\widetilde{g})_x=g^{*} (\iota^{*}(\widetilde{g})_{g_x})= g^{*} (\pi^{*}(g)_{g_x})=(\pi \circ g)^{*}(g)_x=g_x.
\]
Taking into account that 
$
\xi_x=Z(\Psi^{g}(x))=T_{g_x}\iota\cdot Z_{\mathcal{Q}}(g_x),
$
we get $\xi\in \mathfrak{X}^{\perp}(M)$ (for the immersion $\Psi^{g}$)  and $\widetilde{g}(\xi, \xi)=0$.
In order to see that $A_{\xi}=-\mathrm{Id}$, recall that the condition $\mathcal{L}_{Z}\widetilde{g}= 2 \widetilde{g}$ is equivalent to
$$
\widetilde{g}(\widetilde{\nabla}_X Z, Y)+\widetilde{g}(X, \widetilde{\nabla}_Y Z)= 2\widetilde{g}(X,Y), \quad X,Y \in \mathfrak{X}(\widetilde{M}).
$$
In particular, for vector fields $V, W\in \mathfrak{X}(M)$ we have
$$
\widetilde{g}(\widetilde{\nabla}_{T\Psi^{g} \cdot V} Z, T\Psi^{g} \cdot W)+\widetilde{g}(T\Psi^{g}\cdot V, \widetilde{\nabla}_{T\Psi^{g}\cdot W} Z)= 2\widetilde{g}(T\Psi^{g} \cdot V,T\Psi^{g} \cdot W),
$$
and from the polarization identity we arrive to 
$$
\widetilde{g}\Big(\widetilde{\nabla}_{T\Psi^{g} \cdot V} Z, T\Psi^{g} \cdot W\Big)=\widetilde{g}(T\Psi^{g} \cdot V, T\Psi^{g}\cdot W).
$$
We are in position to compute $\widetilde{\nabla}_V\xi$ as follows
$$
\widetilde{\nabla}_{V}\xi=\widetilde{\nabla}_{T\Psi^{g} \cdot V}Z=(\widetilde{\nabla}_{T\Psi^{g} \cdot V}Z)^{\top}+(\widetilde{\nabla}_{T\Psi^{g} \cdot V}Z)^{\perp}=T\Psi^{g} \cdot V+\nabla^{\perp}_{V}\xi
$$
and now the assertion $A_{\xi}=-\mathrm{Id}$ is clear.
\end{proof}

\section{A method to construct pre-ambient spaces}\label{ambient metrics}

\noindent Let $(M,c)$ be a Riemannian conformal structure on an $(n\geq 2)$-dimensional manifold $M$.

\begin{definition}\label{060923A}
A smooth $1$-parameter family $\alpha\colon \R \to  \mathcal{T}_{(1,1)}M$  is called  admissible when
\begin{enumerate}
\item $\alpha(\rho)$ is a self-adjoint tensor field with respect to any representative $g\in c$,
\item   $\alpha(0)=\mathrm{Id}$, 
\item there is $\delta >0$ such that $\alpha(\rho)$ is not singular for $|\rho|< \delta$.
\end{enumerate}
\end{definition} 
\noindent Here, the smoothness of $\alpha$ means that for every $V\in \mathfrak{X}(M)$ and $x\in M$, there exists 
$$
\dot{\alpha}(\rho)(V_{x})=\lim_{\varepsilon\to 0}\frac{\alpha(\rho+\varepsilon)(V_{x})- \alpha(\rho)(V_{x})}{\varepsilon}\in T_{x}M.
$$
In particular, we have $\dot{\alpha}(0)\in \mathcal{T}_{(1,1)}M$.

\begin{remark}
{\rm The condition $(3)$ in the above definition can be deleted when $M$ is compact and,
at least locally, $\delta$ always exists, in the general case.
}
\end{remark}

Let us fix a metric $g\in c$ and an admissible  smooth $1$-parameter family $\alpha\colon \R \to  \mathcal{T}_{(1,1)}M$. For every $\rho \in \R$, we define the following symmetric tensor on $M$,
\[
\langle V,W\rangle_{\rho}^{g}=g\Big(\alpha(\rho)(V), W\Big).
\]
Clearly, $\langle \,,\,\rangle_{0}^{g}=g$ and  so $\langle \,,\,\rangle_{\rho}^{g}$ can be seen as a  $1$-parameter deformation of the metric $g$.
Moreover, 
$\langle \,,\,\rangle_{\rho}^{g}$ is positive definite on $M$ for $|\rho|< \delta$.
Henceforth, let us consider the manifold $\widetilde{M}:= B\times M$, where $B:=\R^{+}\times (-\delta, +\delta)$ with coordinates $(t, \rho)$.  This manifold $\widetilde{M}$ can be endowed  with the Lorentzian metric
\begin{equation}\label{220221A}
\widetilde{g}= d(\rho t)\otimes dt+ dt \otimes d(\rho t)+ t^2\langle -,- \rangle_{\rho}^{g}
\end{equation}
and  with the free 
$\R^{+}$-action $\widetilde{\varphi}(\tau, (t,  \rho, x))=(\tau t, \rho, x)$.
The choice of the metric $g\in c$ provides the global trivialization of $\pi\colon \mathcal{Q}\to M$ given by
$$t^{2}g_{x}\in \mathcal{Q}\mapsto (t, x)\in \R^{+}\times M$$ and the following embedding of  $\mathcal{Q}$ in $\widetilde{M}$ at $\rho=0$, 
\begin{equation}\label{130721D}
\iota_{g} \colon \mathcal{Q}\to \widetilde{M}, \quad t^{2}g_{x}\mapsto (t, 0, x).
\end{equation}

A direct computation shows that
$
\iota_{g} \circ \varphi(\tau, t^{2} g_{x})=\widetilde{\varphi}\circ (\mathrm{id}_{\R^{+}}\times \iota_{g})(\tau, t^2 g_{x})=(\tau t, 0, x)
$
. On the other hand, the fundamental vector field $Z\in \mathfrak{X}(\widetilde{M})$ corresponding to the action $\widetilde{\varphi}$ is 
$
Z=t\frac{\partial}{\partial t}
$
and one directly checks that $\mathcal{L}_{Z}\widetilde{g}=2 \widetilde{g}.$  
Finally,  for $t^2 g_{x}\in \mathcal{Q}$ and $\xi, \eta \in T_{t^{2}g_{x}}\mathcal{Q}$, we have
$$
(\iota^{*}_{g}\widetilde{g})_{t^{2}g_{x}}(\xi, \eta)= \widetilde{g}_{(t,0,x)}(T_{t^{2}g_{x}}\iota_{g}\cdot \xi, T_{t^{2}g_{x}}\iota_{g}\cdot \eta)=t^{2}g_{x}(T_{t^{2}g_{x}}\pi \cdot \xi,T_{t^{2}g_{x}}\pi \cdot \eta ).
$$
Hence, $(\widetilde{M}=B\times M,\widetilde{g})$ where the metric $\widetilde{g}$ is given in (\ref {220221A}) is a pre-ambient space for $(M, c)$. 

We have thus led to the following result. 

\begin{proposition}\label{280321D}
Let $(M,c)$ be a Riemannian conformal structure on an $(n\geq 2)$-dimensional manifold $M$. For every choice of a metric $g\in c$ and an admissible  smooth $1$-parameter family $\alpha\colon \R \to  \mathcal{T}_{(1,1)}M$, the manifold $\widetilde{M}=B\times M$ is a pre-ambient space for $(M,c)$.
\end{proposition}

\begin{remark}\label{280221A}
{\rm In the particular case that $\alpha(\rho)=f^{2}(\rho)\,\mathrm{Id}$ with $f\colon (-\delta, +\delta) \to \R$, $f(0)=1$ and $f>0$, the pre-ambient space $(\widetilde{M}, \widetilde{g})$ with metric $\widetilde{g}= d(\rho t)\otimes dt+ dt \otimes d(\rho t)+(t f(\rho))^{2}g$ is a warped product in the terminology of \cite[Chap. 7]{One83}.
}
\end{remark}

\begin{remark}\label{dw}
{\rm The one-form $\omega$ metrically equivalent to the vector field $Z$ is 
$$
\omega=t^2 d\rho + 2t\rho dt,
$$
thus, we have $d \omega=0$. 
}
\end{remark}

%
%

\bigskip

As a Lorentzian manifold, the pre-ambient space $(\widetilde{M},\widetilde{g})$  is timelike orientable, that is, there exists a globally defined timelike vector field, namely,
	\begin{equation}\label{130721A}
	T:=\frac{1}{t}\partial_{t}- \Big(1+\frac{\rho}{t^{2}}\Big)\partial_{\rho}\in \mathfrak{X}(\widetilde{M}),
	\end{equation}
	which satisfies $\widetilde{g}(T,T)=-2$. To be used later, we also introduce the spacelike vector field
	\begin{equation}\label{130721B}
	E:=\frac{1}{t}\partial_{ t}+ \Big(1- \frac{\rho}{t^2}\Big)\partial_{ \rho}\in \mathfrak{X}(\widetilde{M}),
	\end{equation}
	with $\widetilde{g}(E,E)=2$ and $\widetilde{g}(T,E)=0.$
The set of all natural lifts of vector fields $V \in \mathfrak{X}(M)$ to $\mathfrak{X}(\widetilde{M})$ is denoted by $ \mathfrak{L}(M)$.  For a vector field $V \in \mathfrak{X}(M)$,  its lift to $\mathfrak{L}(M)\subset \mathfrak{X}(\widetilde{M})$ is also denoted by $V$. 

As was mentioned in Remark \ref{280221A}, the metrics $\widetilde{g}$ in (\ref{220221A}) are not warped product metrics, in general. Hence, the formulas for the Levi-Civita connection of warped products metrics in \cite[Prop. 7.36]{One83} do not work.

\begin{proposition}\label{220221B}
	The Levi-Civita connection $\widetilde{\nabla}$ of $(\widetilde{M}, \widetilde{g})$ satisfies
	\begin{equation}\label{conexion M tilde}
		\widetilde{\nabla}_{\partial_{t}} \partial_{t}=\widetilde{\nabla}_{\partial_{\rho}} \partial_{\rho}=0, \quad \quad \widetilde{\nabla}_{\partial_{t}}\partial_{ \rho}=\widetilde{\nabla}_{\partial_{\rho}}\partial_{t}=\dfrac{1}{t}\partial_{\rho},
\end{equation}
\begin{equation}\label{240221A}		
		\widetilde{\nabla}_{\partial_{t}}V=\dfrac{1}{t}V,\quad \widetilde{\nabla}_{\partial_{\rho}}V=\frac{1}{2}\alpha(\rho)^{-1}(\dot{\alpha}(\rho)(V)),
\end{equation}	
\begin{equation}\label{150721A}
\widetilde{\nabla}_{V}W\mid_{\iota_{g}(\mathcal{Q})}=-\frac{1}{2t}\widetilde{g}(\dot{\alpha}(0)(V), W)\partial_{t}-\frac{1}{t^2}\widetilde{g}(V,W)\partial_{\rho}+ \nabla^{g}_{V}W,
\end{equation}	
where $V, W\in \mathfrak{L}(M)$.	
\end{proposition}
\begin{proof} A direct consequence of Koszul formula for the Levi-Civita connection of $(\widetilde{M}, \widetilde{g})$ shows $\widetilde{\nabla}_{\partial_{t}} \partial_{t}=\widetilde{\nabla}_{\partial_{\rho}} \partial_{\rho}=0$ and $\nabla_{\partial_{t}}\partial_{\rho}=\dfrac{1}{t}\partial_{\rho}$.
On the other hand, the Koszul formula also implies $ \widetilde{g}(\widetilde{\nabla}_{\partial_{t}}V, \partial_{t})= \widetilde{g}(\widetilde{\nabla}_{\partial_{t}}V, \partial_{\rho})=0$ and $ 2\widetilde{g}(\widetilde{\nabla}_{\partial_{t}}V, W)=\partial_{t}\widetilde{g}(V,W).$ By definition of the metric $\widetilde{g}$,
$$
\partial_{t}\widetilde{g}(V,W)=2t g(\alpha(\rho)(V),W)= \frac{2}{t}\widetilde{g}(V,W),
$$
and then we get $\widetilde{\nabla}_{\partial_{t}}V=\dfrac{1}{t}V.$ In the same manner, we compute
$$
2\widetilde{g}(\widetilde{\nabla}_{\partial_{\rho}}V, W)=\partial_{\rho}\Big(t^{2}g(\alpha(\rho)(V), W)\Big)=t^{2}g(\dot{\alpha}(\rho)(V), W)= \widetilde{g}\Big(\alpha(\rho)^{-1}(\dot{\alpha}(\rho)(V)), W\Big).
$$ 
From (\ref{240221A}), it follows that 
$$\widetilde{g}(\widetilde{\nabla}_{V}W, \partial_{t})=- \widetilde{g}(\widetilde{\nabla}_{V}\partial_{t}, W)=-\frac{1}{t}\widetilde{g}(V,W), \quad\widetilde{g}(\widetilde{\nabla}_{V}W, \partial_{\rho})=-\frac{1}{2}\widetilde{g}\Big(\alpha(\rho)^{-1}(\dot{\alpha}(\rho)(V)),W\Big).$$ 
In order to compute $\widetilde{g}(\widetilde{\nabla}_{V_{p}}W, U_{p})$ for $U\in \mathfrak{L}(M)$ and $p=(t, 0,x)\in \iota_{g}(\mathcal{Q})$, we can assume $U,V,W\in \mathfrak{L}(M)$ so that all their brackets are zero at the point $p$. Then, the  Koszul formula yields
\[\begin{split}
 2\widetilde{g}(\widetilde{\nabla}_{V_{p}}W, U_{p}) & =V_{p}\, \widetilde{g}(W, U)+W_{p}\, \widetilde{g}(V, U)- U_{p}\, \widetilde{g}(V, W) \\
 & = t^{2}\Big(V_{x}\, g(W, U)+W_{x}\, g(V, U)- U_{x}\, g(V, W) \Big)\\
 & = 2t^{2}g(\nabla^{g}_{V_{x}}W, U_{x})= 2\widetilde{g}\Big(\left(\nabla^{g}_{V}W\right)_{p}, U_{p}\Big).
\end{split}
\]
Therefore, we conclude that
\[
\begin{split}
\widetilde{\nabla}_{V}W\mid_{\iota_{g}(\mathcal{Q})} & = -\frac{1}{2}\widetilde{g}(\widetilde{\nabla}_{V}W, T)T+ \frac{1}{2}\widetilde{g}(\widetilde{\nabla}_{V}W, E)E+ \nabla^{g}_{V}W \\
 & = -\frac{1}{2t}\widetilde{g}(\dot{\alpha}(0)(V), W)\partial_{t}-\frac{1}{t^2}\widetilde{g}(V,W)\partial_{\rho}+ \nabla^{g}_{V}W.
\end{split}
\]
\end{proof}

\begin{remark}\label{II}
{\rm Let us fix $(t,\rho)\in B$ and consider the spacelike submanifold
$$\mathcal{F}\vcentcolon=\{(t,\rho)\}\times M\subset\widetilde{M}.$$
The vector fields $T |_{\mathcal{F}}$ and $E |_{\mathcal{F}}$ span the normal bundle of $\mathcal{F}$ and 
Proposition \ref{220221B} implies 
$$
\widetilde{\nabla}_{V}T |_{\mathcal{F}}= \frac{1}{t^2}V-\frac{1}{2}\Big(1+ \frac{\rho}{t^2}\Big)\alpha(\rho)^{-1}(\dot{\alpha}(\rho)(V))\,\, \mathrm{ and }\,\,
$$
$$
\widetilde{\nabla}_{V}E|_{\mathcal{F}}= \frac{1}{t^2}V+\frac{1}{2}\Big(1- \frac{\rho}{t^2}\Big)\alpha(\rho)^{-1}(\dot{\alpha}(\rho)(V)),
$$
for every $V \in \mathcal{L}(M).$
Therefore, the second fundamental form $\mathrm{II}_{\mathcal{F}}$ is given by
\begin{equation}\label{280221B}
\mathrm{II}_{\mathcal{F}}(V,W)=-\frac{1}{2t}\widetilde{g}\big(\alpha(\rho)^{-1}(\dot{\alpha}(\rho)(V)), W\big)\partial_{t} - \frac{1}{t^2}\Big(\widetilde{g}(V,W)- \rho \widetilde{g}\big(\alpha(\rho)^{-1}(\dot{\alpha}(\rho)(V)), W\big)\Big)\partial_{\rho}
\end{equation}
where $V, W \in \mathfrak{X}(M).$ 
Thus, on the contrary to the warped products metrics, the fibers $\mathcal{F}$ are not totally umbilical,  in general. It is not difficult to show that for a fixed $(t, \rho)$, the corresponding fiber $\mathcal{F}$ is totally umbilical if and only if the endomorphism field $\alpha(\rho)^{-1}\circ \dot{\alpha}(\rho)=f\mathrm{Id}$ for some $f\in\mathcal{C}^{\infty}(M)$.
}
\end{remark}

\begin{remark}
{\rm
For $\alpha(\rho)=f^{2}(\rho)\, \mathrm{Id}$ with $f>0$, the metric $\widetilde{g}$ is a warped metric with  warping function $h(t, \rho)= t f(\rho)$. In this case (\ref{280221B}) reduces to
$$
\mathrm{II}_{\mathcal{F}}(V,W)=-\frac{\widetilde{g}(V,W)}{t f(\rho)}\Big(f'(\rho)\partial_{t}+\frac{f(\rho)- 2 \rho f'(\rho)}{t}\partial_{\rho} \Big).
$$
A direct computation shows that the above formula agrees with \cite[Prop.  7. 35 (3)]{One83}.}
\end{remark}

From \cite{Cap}, the Ricci tensor $\widetilde{\mathrm{Ric}}$ of any pre-ambient space $(\widetilde{M}, \widetilde{g})$ 
restricted to $\iota_{g}(\mathcal{Q})$ satisfies
\begin{equation}\label{120721A}
\widetilde{\mathrm{Ric}}|_{\iota_{g}(\mathcal{Q})}(\partial_{t},\partial_{t})=\widetilde{\mathrm{Ric}}|_{\iota_{g}(\mathcal{Q})}(\partial_{t},V)=0, \quad V\in\mathcal{L}(M)
\end{equation}
if and only if $d\omega|_{\iota_{g}(\mathcal{Q})}=0$. As consequence of Remark \ref{dw}, this  formula (\ref{120721A}) holds for the metric $\widetilde{g}$ in (\ref{220221A}).
The following result provides the other component of $\widetilde{\mathrm{Ric}}$ on $\iota_{g}(\mathcal{Q})$.

\begin{corollary}\label{ricci}
The Ricci tensor $\widetilde{\mathrm{Ric}}$ of $(\widetilde{M}, \widetilde{g})$ satisfies
\begin{equation}\label{220721A}
\widetilde{\mathrm{Ric}}|_{\iota_{g}(\mathcal{Q})}(V,W)=\mathrm{Ric}^g(V,W)-\dfrac{\mathrm{trace}(\dot{\alpha}(0))}{2}g(V,W)-\left(\dfrac{n-2}{2}\right)g\left(\dot{\alpha}(0)(V),W\right),
\end{equation}
where $V, W\in \mathfrak{L}(M)$. For $\xi,\eta\in\mathfrak{X}(\mathcal{Q})$, we have
\begin{itemize}
\item If  $n=2$,
$$\widetilde{\mathrm{Ric}}|_{\iota_{g}(\mathcal{Q})}(T\iota_{g}\cdot \xi,T\iota_{g}\cdot \eta)= 0$$ if and only if $\mathrm{trace}(\dot{\alpha}(0))=2K^g$,  where $K^g$ is the Gauss curvature of  $g$.
\item If $n\geq 3$, 
$$\widetilde{\mathrm{Ric}}|_{\iota_{g}(\mathcal{Q})}(T\iota_{g}\cdot \xi,T\iota_{g}\cdot \eta)= 0$$  if and only if $g(\dot{\alpha}(0)(-), -)=2P^g$,  where $P^g$ is the Schouten tensor of $g$.
\end{itemize}
\end{corollary}
\begin{proof}
Let $(e_{1}, \dots, e_{n})$ be an orthonormal local frame on $(M,g)$ and consider the orthonormal local frame for  $(\widetilde{M},\widetilde{g})$ on  $\rho=0$ given by
 $$
 \left(\frac{1}{\sqrt{2}}T,\frac{1}{\sqrt{2}}E,E_{1}, \dots , E_{n}\right),
 $$ 
 where $E_i=\frac{1}{t}e_i$ and the vector fields $T,E$ are given in (\ref{130721A}) and (\ref{130721B}), respectively.  Then, we get
\[
\begin{split}
\widetilde{\mathrm{Ric}}|_{\iota_{g}(\mathcal{Q})}(V,W) & =\displaystyle\sum_{i=1}^n\widetilde{g}\left(\widetilde{\mathrm{R}}(E_i,V)W,E_i\right)+\dfrac{1}{2}\widetilde{g}\left(\widetilde{\mathrm{R}}(E,V)W,E\right)-\dfrac{1}{2}\widetilde{g}\left(\widetilde{\mathrm{R}}(T,V)W,T\right)\\
& = \displaystyle\sum_{i=1}^n\widetilde{g}\left(\widetilde{\mathrm{R}}(E_i,V)W,E_i\right)+\dfrac{1}{t}\left(\widetilde{g}\left(\widetilde{\mathrm{R}}(\partial_{ t},V)W,\partial_{\rho}\right)+\widetilde{g}\left(\widetilde{\mathrm{R}}(\partial_{\rho},V)W,\partial_{t}\right)\right).
\end{split}
\]
For every vector field $X\in\mathfrak{X}(\widetilde{M})$, we have the following decomposition
\[
\begin{split}
X & = \displaystyle\sum_{i=1}^n f_iE_i+\dfrac{1}{2}\widetilde{g}\left(X,E\right)E-\dfrac{1}{2}\widetilde{g}\left(X,T\right)T \\
 & = \displaystyle\sum_{i=1}^n f_iE_i+\dfrac{1}{t}\left(\widetilde{g}\left(X,\partial_{t}\right)\partial_{\rho}+\widetilde{g}\left(X,\partial_{\rho}\right)\partial_{ t}\right)-\dfrac{2\rho}{t^2}\widetilde{g}\left(X,\partial_{\rho}\right)\partial_{\rho}
\end{split}
\]
where $f_i\in\mathcal{C}^{\infty}(\widetilde{M})$. Let us note that $f_i|_{\rho=0}=\widetilde{g}(X,E_i)$. Now, a straightforward  computation from Proposition \ref{220221B} gives
\begin{equation}\label{130721C}
\widetilde{g}\left(\widetilde{\mathrm{R}}(\partial_{ t},V)W,\partial_{\rho}\right)+\widetilde{g}\left(\widetilde{\mathrm{R}}(\partial_{\rho},V)W,\partial_{t}\right)=0.
\end{equation}
Finally,  it is a standard computation,  from Proposition \ref{220221B} and (\ref{130721C}), to check that
\[
\begin{split}
\widetilde{\mathrm{Ric}}|_{\iota_{g}(\mathcal{Q})}(V,W) & = \sum_{i=1}^n g\left(\nabla^g_{e_i}\nabla^g_V W,e_i\right)-\sum_{i=1}^n g\left(\nabla^g_{V}\nabla^g_{e_i} W,e_i\right)-\displaystyle\sum_{i=1}^n g\left(\nabla^g_{[e_i,V]}W,e_i\right)\\
 & -\dfrac{1}{2}g(V,W)\displaystyle\sum_{i=1}^n g(\dot{\alpha}(0)(e_i),e_i)-\dfrac{n}{2}g(\dot{\alpha}(0)(V),W)\\
 & + \dfrac{1}{2}\sum_{i=1}^n g(e_i,W)g(\dot{\alpha}(0)(V),e_i) +\dfrac{1}{2}\displaystyle\sum_{i=1}^n g(V,e_i)g(\dot{\alpha}(0)(e_i),W)\\
 & = \mathrm{Ric}^g(V,W)-\dfrac{\mathrm{trace}(\dot{\alpha}(0))}{2}g(V,W)-\left(\dfrac{n-2}{2}\right)g(\dot{\alpha}(0)(V),W).
\end{split}
\]
The vanishing properties of the Ricci tensor on $\iota_{g}(\mathcal{Q})$ are direct consequences of  $(\ref{120721A}
)$ and (\ref{220721A}).
\end{proof}

\section{Constructing M\"{o}bius structures from spacelike immersions}\label{main results}

From now on, we assume $(M,c)$ is a Riemannian conformal structure on an $(n\geq 2)$-dimensional manifold $M$ and we have fixed 
\begin{enumerate}
\item a metric $g\in c$ and   
\item an admissible  smooth $1$-parameter family $\alpha\colon \R \to  \mathcal{T}_{(1,1)}M$. 
\end{enumerate}
Thus, we have the pre-ambient space $(\widetilde{M}, \widetilde{g})$ as  in Proposition \ref{280321D}.

\smallskip

For every $u\in C^{\infty}(M)$,  the spacelike immersion $\Psi^{e^{2u}g}$ in Lemma \ref{endomorfismo} satisfies
\begin{equation}\label{immer}
\Psi^{e^{2u}g}: M\to (\widetilde{M}, \widetilde{g}), \quad x\mapsto (e^{u(x)},0, x)
\end{equation}
and $(\Psi^{e^{2u}g})^{*}(\widetilde{g})= e^{2u}g.$
For simplicity of notation,  from now on,  we write $\Psi^{u}$ instead of $\Psi^{e^{2u}g}$.
The differential map of $\Psi^{u}$ is 
 \begin{equation}\label{220221C}
 T\Psi^{u}\cdot V=V(u)e^{u}\partial_{t} |_{\Psi^{u}}+V |_{\Psi^{u}},
 \end{equation}
 where $V\in \mathfrak{X}(M).$
A direct computation from (\ref{220221C}) shows that the vector fields 
\begin{equation}\label{normal vectors}
\xi^u=e^{u}\partial_{t} |_{\Psi^{u}} \quad  \mathrm{and} \quad  \eta^u= e^{-u}\frac{\| \nabla^g u \|^{2}_{g}}{2}\partial_{t} |_{\Psi^{u}}- e^{-2u}\partial_ {\rho} |_{\Psi^{u}}+e^{-2u} \nabla^{g}u |_{\Psi^{u}}
\end{equation}
span the normal bundle of $\Psi^{u}$ and one easy checks that $\{\xi^{u}, \eta^{u}\}$ is a global  lightlike normal frame. 
The lightlike normal vector field $\xi^{u}$ agrees with $Z|_{\Psi^{u}}$ where $Z\in \mathfrak{X}(\widetilde{M})$ is the fundamental vector field  corresponding to the action $\widetilde{\varphi}$.


\begin{lemma}\label{270221A}
Let $\Psi^{u}: M\to (\widetilde{M}, \widetilde{g})$ be the immersion given in $(\ref{immer})$. For every $V \in \mathfrak{L}(M)\subset  \overline{\mathfrak{X}}(M)$, the following formulas hold
$$
(V |_{\Psi^{u}})^{\top}=T\Psi^{u} \cdot V, \quad  (\partial_{\rho}|_{\Psi^{u}})^{\top}=T\Psi^{u}\cdot \nabla^{g}u.
$$
\end{lemma}
 
\begin{proof}
From (\ref{220221C}) and (\ref{normal vectors}),  it is easy to check that
\[
\begin{split}
(V |_{\Psi^{u}})^{\top}& =V |_{\Psi^{u}}+ \widetilde{g}\Big(V |_{\Psi^{u}}, \xi^{u}\Big)\eta^{u}+ \widetilde{g}\Big(V |_{\Psi^{u}}, \eta^{u}\Big)\xi^{u}\\
& = V |_{\Psi^{u}}+ V(u)e^{u}\partial_{t} |_{\Psi^{u}}=T\Psi^{u}\cdot V.
\end{split}
\]
The same proof works for  $(\partial_{\rho}|_{\Psi^{u}})^{\top}$.
\end{proof}

\begin{proposition}\label{lema}
	Let $A_{\xi^u},A_{\eta^u}$ be the Weingarten endomorphisms associated to the lightlike normal vector fields $\xi^u,\eta^u$ given in $(\ref{normal vectors})$, then
	$A_{\xi^u}= -\mathrm{Id}$
	and
	\begin{equation}\label{A2}
	A_{\eta^u}=e^{-2u}\left[\frac{\dot{\alpha}(0)-\|\nabla^g u\|_g^2 \, \mathrm{Id}}{2}+g(\nabla^g u,\mathrm{Id})\nabla^g u\ -\nabla^{g}\nabla^{g} \,u\right],
	\end{equation}
	where $\nabla^{g}\nabla^{g} \,u (V):=\nabla^g_V\nabla^g u$ for all $V\in\mathfrak{X}(M)$.
\end{proposition}
\begin{proof}
	The first assertion is a direct consequence of Lemma \ref{endomorfismo}.  On the other hand, according again to (\ref{220221C}) and  Proposition \ref{220221B},  we have for $V \in \mathfrak{L}(M),$
	\begin{equation}\label{230221A}
	\widetilde{\nabla}_{V}\Big(e^{-u}\frac{\| \nabla^g u \|^{2}_{g}}{2}\partial_{t}\Big)=V\Big(e^{-u}\frac{\| \nabla^g u \|^{2}_{g}}{2}\Big)\partial_{t} |_{\Psi^{u}}+ e^{-2u}\frac{\| \nabla^g u \|^{2}_{g}}{2}\,V |_{\Psi^{u}} ,
	\end{equation}
\begin{equation}\label{230221B}
	\widetilde{\nabla}_V (e^{-2u}\partial_ {\rho})=- 2e^{-2u}V(u)\,\partial_ {\rho} |_{\Psi^{u}}+\frac{e^{-2u}}{2}\Big(\dot{\alpha}(0)(V)\Big)|_{\Psi^{u}}
	\end{equation}
	and 
	\begin{equation}\label{230221C}
	\widetilde{\nabla}_V (e^{-2u}\nabla^{g}u)=- 2e^{-2u}V(u)\,\nabla^{g}u |_{\Psi^{u}}+e^{-2u}\Big(\widetilde{\nabla}_{V}\nabla^{g}u\Big)|_{\Psi^{u}}.
	\end{equation}
Taking into account that $(\partial_{t} |_{\Psi^{u}})^{\top}=0$, from Lemma \ref{270221A},  we also get
$$
\Big(\partial_{\rho} |_{\Psi^{u}}-\nabla^{g}u |_{\Psi^{u}}\Big)^{\top}=0.
$$
Then, from (\ref{230221A}), (\ref{230221B}) and (\ref{230221C}), we arrive to
$$
\left(\widetilde{\nabla}_V\,\eta^u\right)^{\top}=e^{-2u}\left[\dfrac{\|\nabla^{g} u\|_{g}^{2}}{2}(V |_{\Psi^{u}})^{\top}-\frac{1}{2}\Big(\Big(\dot{\alpha}(0)\left(V\right) \Big) |_{\Psi^{u}}\Big)^{\top}+\left(\left(\widetilde{\nabla}_V\nabla^{g} u\right) |_{\Psi^{u}}\right)^{\top}\right].
$$
Now, the proof ends by means of a straightforward computation from (\ref{150721A}) and Lemma \ref{270221A}.
\end{proof}

\begin{corollary}\label{240321A}
Let $\Psi^{u}: M\to (\widetilde{M}, \widetilde{g})$ be the immersion given in $(\ref{immer})$.  The normal vector fields $\xi^{u}$ and $\eta^{u}$ are parallel with respect to the  normal connection.  In particular, the normal curvature tensor vanishes,  that is,  $R^{\perp}(V,W)=0$ for every $V, W \in \mathfrak{X}(M).$
\end{corollary}
\begin{proof}
From Proposition \ref{lema}, we know that $A_{\xi^{u}}=-\mathrm{Id}$. 
Then, the Weingarten formula reads as follows
$$
\widetilde{\nabla}_V\xi^{u}=T \Psi^{u}\cdot V+\nabla^{\perp}_V\,\xi^{u}=V(u)e^{u}\partial_{t} |_{\Psi^{u}}+V |_{\Psi^{u}}+\nabla^{\perp}_V\,\xi^{u}.
$$
On the other hand,  from (\ref{240221A}),  we get 
$$\widetilde{\nabla}_V\xi^{u}=\widetilde{\nabla}_V\, (e^{u}\partial_{t} )=V(u)e^{u}\partial_{t} |_{\Psi^{u}}+e^{u}e^{-u}\,V |_{\Psi^{u}}=V(u)e^{u}\partial_{t} |_{\Psi^{u}}+V |_{\Psi^{u}},$$
and therefore $\nabla^{\perp}_V\,\xi^{u}=0$. 
Now, taking into account that $\{\xi^{u}, \eta^{u}\}$ is a global  lightlike normal frame,  we have
$V \widetilde{g}(\xi^{u}, \eta^{u})=\widetilde{g}(\xi^{u}, \nabla^{\perp}_{V}\eta^{u})=0$  for every $V\in \mathfrak{X}(M)$.  Thus, since $\Psi^{u}$ is a codimension two spacelike submanifold, there is a smooth function $f\in C^{\infty}(M)$ such that $\nabla^{\perp}_{V}\eta^{u}=f\, \xi^{u}$ and then $0=\widetilde{g}(\eta^{u}, \nabla^{\perp}_{V}\eta^{u})=-f$ and so $\nabla^{\perp}_V\,\eta^{u}=0$.
\end{proof}

\begin{remark}\label{100321A}
{\rm From Proposition \ref{lema} and formula (\ref{270221B}),  one obtains the second fundamental form $\mathrm{II}^{u}$ of $\Psi^{u}$ as follows
$$
\mathrm{II}^{u}(V,W)=-g\Big(\frac{\dot{\alpha}(0)(V)-\|\nabla^g u\|_g^2 \, V}{2}+V(u)\nabla^g u\ -\nabla^{g}_{V}\nabla^{g} \,u,W\Big)\xi^{u}+e^{2u}g(V,W)\eta^{u},
$$
for every $V, W \in \mathfrak{X}(M).$ 
In particular,  the corresponding mean curvature vector field is
\begin{equation}\label{020321A}
\mathbf{H}^{u}=\frac{e^{-2u}}{n}\Big( \triangle^{g} u-\frac{\mathrm{trace}(\dot{\alpha}(0))-(n-2) \|\nabla^g u\|_g^2}{2} \Big)\xi^{u}+\eta^{u},
\end{equation}
where $\triangle^{g}$ denotes the Laplace operator of the metric $g$. 
}
\end{remark}

Now, we are in position to state the main result of this paper.
Assume $(M,c)$ is a Riemannian conformal structure on an $(n\geq 2)$-dimensional manifold $M$ and $\alpha\colon \R \to  \mathcal{T}_{(1,1)}M$ is an admissible  smooth $1$-parameter family. By means of Proposition \ref{lema}  we have $
A_{\eta^0}=\frac{\dot{\alpha}(0)}{2} 
$  and then for every $u\in \mathcal{C}^{\infty}(M)$,  
$$
A_{\eta^{u}}=e^{-2u}\left[A_{\eta^{0}}-\frac{1}{2}\|\nabla^g u\|_g^2 \, \mathrm{Id}+g(\nabla^g u,\mathrm{Id})\nabla^g u\ -\nabla^{g}\nabla^{g} \,u\right].
$$
Hence, for every $V,W\in \mathfrak{X}(M)$ we get
$$
e^{2u}g\left(A_{\eta^u}(V),W \right)=g(A_{\eta^{0}}(V), W)- \frac{\| \nabla^g u \|^{2}_{g}}{2}g- \mathrm{Hess}^{g}(u)+ du \otimes du.
$$
In other words, the assignment
		\begin{equation}\label{120123A}
		D:c\to \mathcal{T}_{(0,2)}M, \quad e^{2u}g\mapsto e^{2u}g\left(A_{\eta^u}(-),- \right),
		\end{equation}
satisfies the conformal transformation law $(2)$ in  \ref{280221C}. In addition, if we assume $\mathrm{trace}_{g}(A_{\eta^{0}})=\frac{\mathrm{scal}^{g}}{2(n-1)}$, the map $D$ defines a   M\"{o}bius structure for the Riemannian conformal structure $(M, c)$.	
Therefore, we have obtained the following result.

\begin{theorem}\label{040222}
Let $(M,c)$ is a Riemannian conformal structure on an $(n\geq 2)$-dimensional manifold $M$. Assume the admissible  smooth $1$-parameter family $\alpha\colon \R \to  \mathcal{T}_{(1,1)}M$ satisfies $\mathrm{trace}(\dot{\alpha}(0))=\frac{\mathrm{scal}^{g}}{n-1}$.
Then, the assignment $D$ in (\ref{120123A})
		defines
 a  M\"{o}bius structure for the Riemannian conformal structure $(M, c)$. 
\end{theorem}

Conversely,  every M\"{o}bius structure  $(M, c, D)$ can be constructed (at least locally) from the above Theorem.
In fact, fix $g\in c$ and  consider $$\alpha(\rho)= \mathrm{Id}+ 2 \rho \,\widehat{D}(g),$$  where
$D(g)(V, W)=g(\widehat{D}(g)(V),W)$ for $V, W \in \mathfrak{X}(M)$. For any $x\in M$, there is  an open subset $x\in \mathcal{O}\subset M$ such that $\alpha$ is an admissible  smooth $1$-parameter family on $\mathcal{T}_{(1,1)}\mathcal{O}$. It is easily checked
 $(\mathcal{O},c, D)$ is obtained from  $\alpha$ by means of Theorem \ref{040222}.
Note  that  $\alpha(\rho)= \mathrm{Id}+ 2 \rho \,\widehat{D}(g)$ can be replaced for any curve with $\alpha(0)=\mathrm{Id}$ and $\dot{\alpha}(0)=2\,\widehat{D}(g).$

\begin{remark}
{\rm When $M$ is compact,  every M\"{o}bius structure  $(M, c, D)$ is globally recovered from suitable Weingarten endomorphisms
as in Theorem \ref{040222}.
}

\end{remark}

\begin{corollary}\label{080321A}
Let $(M,g)$ be a Riemannian manifold with $\mathrm{dim}\,M\geq 3$. Then
the Schouten tensor  $P^g$ is given by $P^g=g(A(-), -)$  (at least locally) where $A$ is 
the Weingarten endomorphism of a suitable isometric codimension two immersion of $(M,g)$ in a Lorentzian manifold $(\widetilde{M}, \widetilde{g}).$
\end{corollary}

%


\begin{remark}
{\rm This result could be compared with the classical Brinkmann result \cite{Bri} in the 1920s which stated that an $(n\geq 3)$-dimensional  simply connected Riemannian manifold is (locally) conformally flat if and only if it can be isometrically immersed in the future lightlike cone $\mathcal{N}^{n+1}\subset \L^{n+2}$. This classical result is presented in a modern form in \cite{AD}.
}
\end{remark}

\begin{remark}
{\rm Since, there is no preferred  M\"{o}bius structure on a $2$-dimensional Riemannian conformal structure, Theorem \ref{040222} provides an explicit method to construct such structures. Moreover, by means of  Corollary $\ref{ricci}$, the condition $\mathrm{trace}(\dot{\alpha}(0))=2 K^g$, where $K^g$ is the Gauss curvature of fixed metric $g$ implies that the Ricci tensor of $\widetilde{g}$ satisfies 
$
\widetilde{\mathrm{Ric}}|_{\iota_{g}(\mathcal{Q})}(T\iota_{g}\cdot \xi,T\iota_{g}\cdot \eta)= 0
$ 
for all $\xi,\eta\in\mathfrak{X}(\mathcal{Q})$.}
\end{remark}

\begin{remark}\label{k=H}
{\rm Under the assumption $\mathrm{trace}(\dot{\alpha}(0))=\frac{\mathrm{scal}^g}{n-1}$ and  by means of the  relationship between the scalar curvature of conformally related metrics, formula (\ref{020321A}) reduces to
$$
\mathbf{H}^{u}=-\dfrac{1}{2n(n-1)}\mathrm{scal}^{e^{2u}g}\,\xi^u+\eta^u,
$$
and therefore, $\|\mathbf{H}^u\|^2=\frac{\mathrm{scal}^{e^{2u}g}}{n(n-1)}.$ This formula widely generalizes \cite[Cor. 4.5]{PPR} and \cite[Cor. 3.7]{PR13}. Therefore, the causality of $\mathbf{H}^u$ is determined by the sign of $\mathrm{scal}^{e^{2u}g}$. For conformal Riemannian structures on compact $2$-dimensional manifolds $(M, c)$ and, as direct consequence of the Gauss-Bonnet theorem, we get
$$
\int_{M}e^{2u} \|\mathbf{H}^u\|^2 \, d\mu_{g}= 2\pi \chi(M),
$$
where $\chi(M)$ is the Euler characteristic of the manifold $M$ and $d\mu_{g}$ is the canonical measure associated to $g$.
Also, from Corollary \ref{240321A}, the condition $\nabla^{\perp}\mathbf{H}^{u}=0$ is equivalent to $\mathrm{scal}^{e^{2u}g}$ being constant (compare with \cite[Cor. 3.10]{PR13}). The positive solution to the Yamabe problem states that on every conformal Riemannian structure $(M,c)$ on a compact manifold $M$ there is a metric $g\in c$ with constant scalar curvature. Therefore, in the compact case, there exists an immersion $\Psi^{u}$ as in  (\ref{immer}) with parallel mean curvature vector field.
}
\end{remark}

\section{An Application}\label{app}

\noindent For a  M\"{o}bius structure $(M, c, D)$ on a $2$-dimensional manifold $M$, the Cotton-York tensor for $g\in c$ has been introduced in \cite{calderbank} and \cite{burden2015numerical2}
 as follows
\begin{equation}\label{150721B}
C(g)(U,V,W)=g\left(\left(\nabla^{g}_{U}\widehat{D}(g)\right)(V)-\left(\nabla^{g}_{V}\widehat{D}(g)\right)(U),W\right), \quad U, V, W\in \mathfrak{X}(M).
\end{equation}
This definition formally agrees with the usual Cotton-York tensor defined from the Schouten tensor of an $(n\geq 3)$-dimensional  Riemannian manifold  $(M,g)$.
The Cotton-York tensor given in (\ref{150721B}) for $n=2$ satisfies $C(g)=C(e^{2u}g)$  (e.g., \cite{burden2015numerical2}).

In this Section, we assume  $(M, c, D)$ is a  M\"obius structure on a $2$-dimensional manifold $M$ which is achieved by means of Theorem \ref{040222}.
\begin{lemma}\label{160721C}
Let  $(M, c, D)$ be a  M\"obius structure on a $2$-dimensional manifold $M$. Then,  the Cotton-York tensor  satisfies
$$
C(g)(V,U, W)\xi^{u}=(\nabla_U \mathrm{II}^{u})(V,W)-(\nabla_V \mathrm{II}^{u} )(U,W),
$$
for 
$
\Psi^{u}$
as in $(\ref{immer})$. Hence, the Codazzi equation $(\ref{160721B})$ reduces to
$$
\left(\widetilde{R}(T\Psi^{u}\cdot U,T\Psi^{u}\cdot V)T\Psi^{u}\cdot W\right)^{\perp}= C(g)(V,U, W)\xi^{u}.
$$
\end{lemma}
\begin{proof}
According to Remark \ref{100321A},
the second fundamental form of $\Psi^{u}$ is
\begin{equation}\label{300321A}
\mathrm{II}^{u}(V,W)=-P(e^{2u}g)(V,W)\xi^{u}+e^{2u}g(V,W)\eta^{u}. 
\end{equation}
From Corollary \ref{240321A},  we have
$
\nabla^{\perp}_{U}\, \xi^{u}=\nabla^{\perp}_{U}\, \eta^{u}=0$ 
and then, a direct computation gives 
$$
\nabla^{\perp}_{U} (\mathrm{II}^{u}(V,W))=- e^{2u}g\Big(\left(\nabla^{e^{2u}g}_{U}\widehat{P}(e^{2u}g)\right)(V), W\Big)\xi^{u}.
$$
Now, the covariant derivative of the second fundamental form in (\ref{160721A}) is easily computed. The proof ends by means of (\ref{150721B}) and $C(g)=C(e^{2u}g)$ for $n=2$.
\end{proof}

\begin{definition}(\cite{calderbank}, \cite{burden2015numerical2})
	A M\"obius structure $(M, c, D)$ on a $2$-dimensional manifold $M$ is called flat when $C(g)=0$ for every $g\in c$.
	\end{definition}
	
	As a direct consequence of Lemma \ref{160721C}, we have.
	
\begin{proposition}\label{290321D}
A M\"obius structure $(M, c, D)$ on a $2$-dimensional manifold $M$ is flat if and only if for every immersion 
$
\Psi^{u}: M\to (\widetilde{M}, \widetilde{g})
$
as in $(\ref{immer})$,
the curvature tensor $\widetilde{R}$ of the pre-ambient manifold $(\widetilde{M},\widetilde{g})$ satisfies
$$
\widetilde{R}(T\Psi^{u}\cdot U,T\Psi^{u}\cdot V)T\Psi^{u}\cdot W\in\mathfrak{X}(M)\subset \overline{\mathfrak{X}}(M),
$$
 for all $U,V,W\in\mathfrak{X}(M)$.
\end{proposition}

\begin{remark}
{\rm 
For a flat M\"obius structure $(M, c, D)$,  Proposition \ref{290321D} states that tangent spaces of $M$ along $\Psi^{u}$
are invariant under the curvature tensor of $(\widetilde{M}, \widetilde{g}).$
As far as we know, the theory of immersions satisfying this condition appeared for the first time in \cite{Ogiue}.  K. Ogiue called these immersions as invariant immersions.   This condition generalizes properties of the immersions into manifolds of constant sectional curvature.  The existence of curvature invariant tangent subspaces in a general Riemannian manifold is related with the existence of totally geodesic submanifolds (see \cite{Tsukada} for more details).
}
\end{remark}

\end{document}